\documentclass{amsart}
\usepackage[all]{xy}
\usepackage{amssymb} % for \restriction

\mathchardef\mhyp="2D % mathematical hyphen

\newcommand{\Clopen}{\mathsf{Clopen}}
\newcommand{\Open}{\mathsf{Open}}
\newcommand{\Closed}{\mathsf{Closed}}
\newcommand{\Constructible}{\mathsf{Constructible}}
\newcommand{\Borel}{\mathsf{Borel}}
\newcommand{\BP}{\mathsf{BP}}
\newcommand{\Nwd}{\mathsf{Nwd}}
\newcommand{\G}[1]{\mathsf{G_{#1}}}
\newcommand{\RO}{\mathsf{RO}}

\newcommand{\Or}{\mathsf{\mhyp or\mhyp}}
\newcommand{\Meets}{\mathsf{\mhyp meets\mhyp}}

\newcommand{\ClosedOrOpen}{{\Closed\Or\Open}}
\newcommand{\ClosedOrG}[1]{{\Closed\Or\G{#1}}}
\newcommand{\ClosedOrNwd}{{\Nwd\Or\Closed}}
\newcommand{\OpenOrNwd}{{\Nwd\Or\Open}}
\newcommand{\ClosedOrRO}{{\Closed\Or\RO}}
\newcommand{\ClosedMeetsRO}{{\Closed\Meets\RO}}
\newcommand{\NwdOrRO}{{\Nwd\Or\RO}}

\newcommand{\A}{\mathcal A}
\newcommand{\B}{\mathcal B}
\newcommand{\F}{\mathcal F}
\newcommand{\U}{\mathcal U}
\newcommand{\E}{\mathcal E}
\newcommand{\I}{\mathcal I}
\newcommand{\w}{\omega}

\newcommand{\clo}[1]{\overline{#1}}
\newcommand{\symdiff}{\mathbin{\Delta}}
\newcommand{\restr}[1]{\mathbin{\restriction}#1}

\def \int {\operatorname{int}} % topological interior
\def \ro {\operatorname{ro}} % open regularization

\newtheorem{theorem}{Theorem}[section]
\newtheorem{corollary}[theorem]{Corollary}
\newtheorem{proposition}[theorem]{Proposition}
\theoremstyle{definition}
\newtheorem{definition}[theorem]{Definition}
\newtheorem{remark}[theorem]{Remark}
\newtheorem{example}[theorem]{Example}

\newif \iftransposed
\transposedtrue
\newcommand{\tr}[2]{\iftransposed#2\else#1\fi}
\title{Lower separation axioms via Borel and Baire algebras}
\author{Taras Banakh, Adam Barto\v s}
\address{T.~Banakh: Ivan Franko National University of Lviv (Ukraine), and Jan Kochanowski University in Kielce (Poland)}
\email{t.o.banakh@gmail.com}
\address{A.~Barto\v s: Department of Mathematical Analysis, Faculty of Mathematics and Physics,
Charles University, Prague (Czech Republic)}
\email{drekin@gmail.com}


\thanks{The second author was supported by the grants GAUK~970217 and SVV-2017-260456 of Charles University.}
\dedicatory{To the memory of Stoyan Nedev}

\begin{document}
\begin{abstract} Let $\kappa$ be an infinite regular cardinal. We define a topological space $X$ to be
{\em $T_{\kappa\mhyp\Borel}$-space} (resp. a {\em $T_{\kappa\mhyp\BP}$-space}) if for every $x\in X$ the singleton $\{x\}$ belongs to the smallest $\kappa$-additive algebra of subsets of $X$ that contains all open sets (and all nowhere dense sets) in $X$. Each $T_1$-space is a $T_{\kappa\mhyp\Borel}$-space and each $T_{\kappa\mhyp\Borel}$-space is a $T_0$-space. On the other hand, $T_{\kappa\mhyp\BP}$-spaces need not be $T_0$-spaces.

We prove that a topological space $X$ is a $T_{\kappa\mhyp\Borel}$-space (resp. a $T_{\kappa\mhyp\BP}$-space) if and only if for each point $x\in X$ the singleton $\{x\}$ is the intersection of a closed set and a $G_{<\kappa}$-set in $X$ (resp. $\{x\}$ is either nowhere dense or a $G_{<\kappa}$-set in $X$). Also we present simple examples distinguishing the separation axioms $T_{\kappa\mhyp\Borel}$ and $T_{\kappa\mhyp\BP}$ for various infinite cardinals $\kappa$, and we relate the axioms to several known notions, which results in a quite regular two-dimensional diagram of lower separation axioms.
\end{abstract}
\maketitle

\section{Introduction}

In this paper we define and study some separation axioms, weaker than the classical separation axiom $T_1$. First we recall three known definitions.

A topological space $X$ is called
\begin{itemize}
\item a {\em $T_1$-space} if for each $x\in X$ the singleton $\{x\}$ is a closed subset of $X$;
\item a {\em $T_{\frac12}$-space} if for each $x\in X$ the singleton $\{x\}$ is closed or open in $X$;
\item a {\em $T_0$-space} if for each $x\in X$ the singleton $\{x\}$ coincides with the intersection of all open or closed sets containing $x$.
\end{itemize}

$T_{\frac12}$-spaces were introduced by McSherry in \cite{McSherry_74} under name $T_{ES}$. The name $T_{\frac12}$ comes from Levine, who earlier introduced a different but equivalent condition in \cite{Levine_70}.

These known notions suggest the following general definition.

\begin{definition} Let $\A(X)$ be a family of sets of a topological space $X$. We shall say that $X$ is a {\em $T_\A$-space} if for each $x\in X$ the singleton $\{x\}$ belongs to the family $\A(X)$.
\end{definition}

In the role of the family $\A(X)$ we shall consider the following families:
\begin{itemize}
\item the family $\Open(X)$ of all open subsets of $X$, i.e., the topology of $X$;
\item the family $\Closed(X)$ of all closed subsets of $X$;
\item the algebra $\Constructible(X)$ of all constructible subsets of $X$, i.e., the smallest algebra of sets containing the topology of $X$;
\item the $\sigma$-algebra $\Borel(X)$ of all Borel subsets of $X$, i.e., the smallest $\sigma$-algebra of sets containing the topology of $X$;
\item the ideal $\Nwd(X)$ of all nowhere dense subsets of $X$;
\item the $\sigma$-algebra $\BP(X)$ of all sets with the Baire property in $X$, i.e., the smallest $\sigma$-algebra of sets, containing all open and all nowhere dense sets.
\end{itemize}

We recall that a family $\A$ of subsets of a set $X$ is called
\begin{itemize}
\item an {\em algebra} if for any $A,B\in\A$ the sets $A\cap B$, $A\cup B$, $X\setminus A$ belong to $\A$;
\item {\em $\kappa$-additive} for a cardinal $\kappa$ if for any subfamily $\F\subset\A$ of cardinality $|\F|<\kappa$ the union $\bigcup\F$ belongs to $\A$;
\item a {\em $\sigma$-algebra} if $\A$ is an $\w_1$-additive algebra of sets.
\end{itemize}

For a topological space $X$ and an infinite cardinal $\kappa$ let
\begin{itemize}
\item $\kappa$-$\Borel(X)$ be the smallest $\kappa$-additive algebra of subsets of $X$, containing all open sets in $X$;
\item $\kappa$-$\BP(X)$ be the smallest $\kappa$-additive algebra of subsets of $X$, containing all open sets and all nowhere dense sets in $X$;
\item $\infty\mhyp\Borel(X)=\bigcup\limits_{\kappa}\kappa\mhyp\Borel(X)=|X|^+\mhyp\Borel(X)$;
\item $\infty\mhyp\BP(X)=\bigcup\limits_{\kappa}\kappa\mhyp\BP(X)=|X|^+\mhyp\BP(X)$.
\end{itemize}
Here by $|X|$ we denote the cardinality of the space $X$ and by $|X|^+$ the successor of the cardinal $|X|$.

The property of being a $T_\Borel$-space, i.e. being a $T_{\kappa\mhyp\Borel}$-space for $\kappa = \w_1$, was considered by Harley and NcNulty in \cite{Harley_McNulty_79}. They gave several characterizations, examined preservation under subspaces and products, and showed that the property lies strictly between $T_1$ and $T_0$.
In \cite{Lo_01} Lo uses the name $GT_D$ instead of $T_\Borel$ and considers also spaces where every singleton is either closed or $G_\delta$. These spaces are called $GT_{\frac12}$ by Lo. Among other results regarding interactions between these properties and the lattice of all topologies on a particular set, minimal $GT_D$- and minimal $GT_{\frac12}$-spaces are characterized.

\begin{remark} \label{t:regular}
	For an infinite singular cardinal $\kappa$ and a topological space $X$ we
have $\kappa\mhyp\Borel(X) = \kappa^+\mhyp\Borel(X)$ and
$\kappa\mhyp\BP(X) = \kappa^+\mhyp\BP(X)$. In fact, any $\kappa$-additive
algebra is actually $\kappa^+$-additive. Since there is a cofinal set
$\{\alpha_\beta: \beta < \lambda\} \subset \kappa$ for some $\lambda < \kappa$,
we have $\bigcup_{\alpha < \kappa} A_\alpha = \bigcup_{\beta < \lambda}
(\bigcup_{\alpha < \alpha_\beta} A_\alpha)$ for any family $\{A_\alpha: \alpha
< \kappa\} \subset \A$. Therefore, we may often restrict ourselves to regular
$\kappa$.
\end{remark}

%\begin{remark}
%	It indeed holds that $\infty\mhyp\Borel(X) = |X|^+\mhyp\Borel(X)$ and $\infty\mhyp\BP(X) = |X|^+\mhyp\BP(X)$ for every topological space $X$. Let $\A \subset \mathcal P(X)$ be an algebra and let $\kappa$ be the smallest cardinal such that $\A$ is not $\kappa$-additive. Clearly, $\kappa = \lambda^+$ for some regular cardinal $\lambda$, and this is witnessed by some family $\{A_\alpha: \alpha < \lambda\} \subset \A$ such that $\bigcup_{\alpha < \lambda} A_α \notin \A$. Let us consider the family $\B = \{A_\alpha \setminus \bigcup_{\beta < \alpha} A_\beta: \alpha < \lambda\}$. By the minimality of $\lambda$ we have $\B \subset \A$, and we also have $\bigcup\B = \bigcup_{\alpha < \lambda} A_\alpha \notin \A$. Again by the minimality of $\lambda$ and by the disjointness of $\B$ we have $\lambda = |\B| \leq |X|$, and hence $\kappa \leq |X|^+$.
%\end{remark}

Given two families $\A(X)$ and $\B(X)$ we also use the following notation.
\begin{itemize}
	\item $\A\Or\B(X)$ denotes the family $\A(X) \cup \B(X)$;
	\item $\A\Meets\B(X)$ denotes the family $\{A \cap B: A \in \A(X), B \in \B(X)\}$.
\end{itemize}

Using the introduced families and notation, $T_\Open$ means discreteness, $T_\Closed$ is $T_1$, and $T_\ClosedOrOpen$ is $T_\frac{1}{2}$. Also, $T_{\infty\mhyp\Borel}$ is equivalent to $T_0$, which shall be proved in Corollary~\ref{c:T0=TB}.
The inclusion relations between the families are described in the following diagram (in which $\kappa$ is any regular uncountable cardinal and an arrow $\A\to\B$ means that $\A\subset\B$).
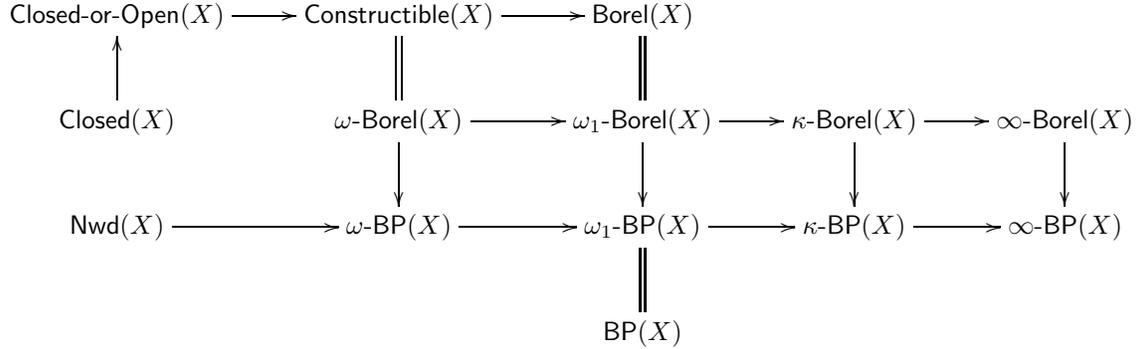
$$
\xymatrix{
\ClosedOrOpen(X)\ar[r]&\Constructible(X)\ar[r]\ar@{=}[d]&\Borel(X)\ar@{=}[d]\\
\Closed(X)\ar[u]&\w\mhyp\Borel(X)\ar[r]\ar[d]&\w_1\mhyp\Borel(X)\ar[r]\ar[d]&\kappa\mhyp\Borel(X)\ar[r]\ar[d]&\infty\mhyp\Borel(X)\ar[d]\\
\Nwd(X)\ar[r]&\w\mhyp\BP(X)\ar[r]&\w_1\mhyp\BP(X)\ar[r]&\kappa\mhyp\BP(X)\ar[r]&\infty\mhyp\BP(X)\\
&&\BP(X)\ar@{=}[u]
}
$$

These inclusion relations imply the following implications between the corresponding separation axioms (holding for any topological space):
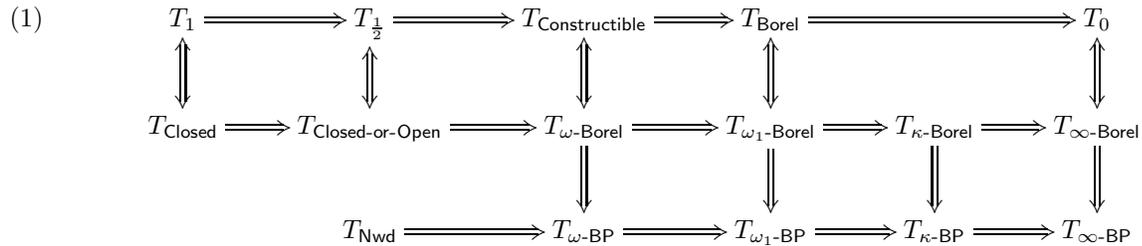
\begin{equation}\label{diag_1}
\xymatrix{
T_1\ar@{=>}[r]&T_{\frac12}\ar@{=>}[r]&T_{\Constructible}\ar@{=>}[r]&T_\Borel\ar@{=>}[rr]&&T_0\\
T_{\Closed}\ar@{<=>}[u]\ar@{=>}[r]&T_{\ClosedOrOpen}\ar@{<=>}[u]\ar@{=>}[r]&T_{\w\mhyp\Borel}\ar@{=>}[r]\ar@{=>}[d]\ar@{<=>}[u]&T_{\w_1\mhyp\Borel}\ar@{=>}[r]\ar@{=>}[d]\ar@{<=>}[u]&T_{\kappa\mhyp\Borel}\ar@{=>}[r]\ar@{=>}[d]&T_{\infty\mhyp\Borel}\ar@{=>}[d]\ar@{<=>}[u]\\
&T_\Nwd\ar@{=>}[r]&T_{\w\mhyp\BP}\ar@{=>}[r]&T_{\w_1\mhyp\BP}\ar@{=>}[r]&T_{\kappa\mhyp\BP}\ar@{=>}[r]&T_{\infty\mhyp\BP}
}
\end{equation}

%\begin{remark} \label{t:T_0}
%The equivalence of the separation axioms $T_0$ and $T_{\infty\mhyp\Borel}$ can be proved as follows. Let $X$ be a topological space and for every $x \in X$ let $U_x = \bigcap\{U: x \in U \in\Open(X)\}$. Observe that $U_x$ is the set of all points from which $x$ cannot be separated, while $\clo{\{x\}}$ is the set of all points that cannot be separated from $x$. Hence, $X$ is $T_0$ if and only if $\{x\} = U_x \cap\clo{\{x\}}$ for every $x \in X$.
%\end{remark}

\section{Characterizations}

Now we give characterizations of separation axioms $T_{\kappa\mhyp\Borel}$ and $T_{\kappa\mhyp\BP}$ for various cardinals $\kappa$. A subset $A$ of a topological space $X$ is defined to be a {\em $G_{<\kappa}$-set} in $X$ if $A=\bigcap\U$ for some family $\U\subset\Open(X)$ of cardinality $|\U|<\kappa$.
We also denote the family of all $G_{<\kappa}$-sets in $X$ by $\G{<\kappa}(X)$ and we denote the union $\bigcup_\kappa \G{<\kappa}(X)$ by $\G{\infty}(X)$. Since a $G_{<\w_1}$-set is also called a $G_\delta$-set, $\G{\delta}(X)$ will be the alternative name for $\G{<\w_1}(X)$.

The following theorem was proved for $\kappa = \w_1$ by Harley and McNulty in \cite{Harley_McNulty_79}. We include the generalized proof for completeness.

\begin{theorem}\label{t:Borel} Let $\kappa$ be an infinite regular cardinal. A topological space $(X,\tau)$ is a $T_{\kappa\mhyp\Borel}$-space if and only if for each $x\in X$ the singleton $\{x\}$ can be written as $\{x\}=F\cap G$ for some closed set $F\subset X$ and some $G_{<\kappa}$-set $G\subset X$.
That is, $T_{\kappa\mhyp\Borel}$ is equivalent to $T_{\Closed\Meets\G{<\kappa}}$.
\end{theorem}

\begin{proof} The ``if'' part is trivial. To prove the ``only if'' part, assume that for some $x\in X$ the singleton $\{x\}$ belongs to the algebra $\kappa$-$\Borel(X)\subset\mathcal P(X)$. Here $\mathcal P(X)$ stands for the power-set of $X$.
We need to prove that the singleton $\{x\}$ belongs to the family of sets $$\mathcal G=\{F\cap G:\mbox{$F$ is a closed set in $X$ and $G$ is a $G_{<\kappa}$-set in $X$}\}.$$

The algebra $\kappa\mhyp\Borel(X)$ can be written as the union $\kappa\mhyp\Borel(X)=\bigcup_{\alpha}\kappa\mhyp\Borel_\alpha(X)$ of an increasing transfinite sequence $\big(\kappa\mhyp\Borel_\alpha(X)\big)_{\alpha}$ of families $\kappa\mhyp\Borel_\alpha(X)\subset\mathcal P(X)$, defined for any ordinal $\alpha$ by the recursive formula:
$$
\kappa\mhyp\Borel_\alpha(X):=\ClosedOrOpen(X)\cup
\big\{\textstyle{\bigcap\F,\bigcup\F}:\F\subset\bigcup_{\beta<\alpha}\kappa\mhyp\Borel_\beta(X),\;|\F|<\kappa\big\}.$$

By transfinite induction, for every ordinal $\alpha$ we shall prove the following statement:
\begin{itemize}
\item[$(*_\alpha)$] for every set $E\in\kappa\mhyp\Borel_\alpha(X)$ containing $x$, there exists a set $G\in \mathcal G$ such that $x\in G\subset E$.
\end{itemize}
This statement is trivial for $\alpha=0$ (as $\kappa\mhyp\Borel_0(X)=\ClosedOrOpen(X)\subset\mathcal G$). Assume that for some ordinal $\alpha$ and all ordinals $\beta<\alpha$ the statements $(*_\beta)$ are proved. Take any set $B\in\kappa\mhyp\Borel_\alpha(X)$, containing $x$. The definition of the family $\kappa\mhyp\Borel_\alpha(X)$ implies that $B$ is equal to $\bigcap\E$ or $\bigcup\E$ for some non-empty family $\E\subset\bigcup_{\beta<\alpha}\kappa\mhyp\Borel_\beta(X)$ of cardinality $|\E|<\kappa$. If $B=\bigcup\E$, then $x\in E\subset B$ for some $E\in\E$. Choose $\beta<\alpha$ such that the family $\kappa\mhyp\Borel_\beta(X)$ contains the set $E$ and using the inductive hypothesis $(*_\beta)$, find a set $G\in\mathcal G$ such that $x\in G\subset E\subset B$.

Next, assume that $B=\bigcap\E$. By the inductive hypothesis, for every $E\in\mathcal E\subset\bigcup_{\beta<\alpha}\kappa\mhyp\Borel_\beta(X)$, there exists a set $G_E\in\mathcal G$ such that $x\in G_E\subset E$. For every $E\in\E$ the set $G_E$ can be written as $G_E=F_E\cap\bigcap\U_E$ for some closed set $F_E\subset X$ and some family $\U_E\subset\tau$ of cardinality $|\U_E|<\kappa$. The regularity of the cardinal $\kappa$ ensures that the family $\U=\bigcup_{E\in\E}\U_E$ has cardinality $|\U|<\kappa$ and hence the set $G:=\bigcap_{E\in\E}F_E\cap \bigcap\U$ belongs to the family $\mathcal G$. It is clear that $x\in G=\bigcap_{E\in\E}G_E\subset\bigcap\E=B$, which completes the proof of the statement $(*_\alpha)$.
\smallskip

If $\{x\}\in\kappa\mhyp\Borel(X)$, then $\{x\}\in \kappa\mhyp\Borel_\alpha(X)$ for some ordinal $\alpha$ and by the statement $(*_\alpha)$, there exists a set $G\in\mathcal G$ such that $x\in G\subset \{x\}$ and hence $\{x\}=G\in\mathcal G$.
\end{proof}

\begin{corollary}\label{c:T0=TB} A topological space $X$ is a $T_0$-space if
and only if $X$ is a $T_{\infty\mhyp\Borel}$-space.
\end{corollary}

\begin{remark}
	Theorem~\ref{t:Borel} implies that a topological space $X$ is a $T_\Constructible$-space if and only if $X$ is a $T_{\Closed\Meets\Open}$-space. So, $X$ is a $T_D$-space in the sense of \cite[Definition~3.1]{Aull_Thron_62}.
	The axiom $T_D$ is also used in point-free topology -- in the class of $T_D$-spaces the spaces can be reconstructed from their lattices of open sets. More precisely, every isomorphism of lattices of open sets of $T_D$-spaces is induced by exactly one homeomorphism between the spaces \cite[Proposition I.2.4]{Picado_Pultr_12}.
\end{remark}

Next, we characterize $T_{\kappa\mhyp\BP}$-spaces.

\begin{theorem}\label{t:BP} Let $\kappa$ be an infinite regular cardinal. A topological space $(X,\tau)$ is a $T_{\kappa\mhyp\BP}$-space if and only if for each $x\in X$ the singleton $\{x\}$ is either nowhere dense or a $G_{<\kappa}$-set in $X$.
That is, $T_{\kappa\mhyp\BP}$ is equivalent to $T_{\Nwd\Or\G{<\kappa}}$.
\end{theorem}

\begin{proof} The ``if'' part is trivial. To prove the ``only if'' part, assume that $\{x\}\in \kappa\mhyp\BP(X)$.

Let $\I$ be the ideal of sets that can be covered by $<\kappa$ many nowhere
dense sets in $X$. The regularity of the cardinal $\kappa$ implies the
$\kappa$-additivity of the ideal $\I$. Let $\A$ denote the family of sets
$A\subset X$ for which there exists an open set $U_A\subset X$ such that the
symmetric difference $A\symdiff U_A$ belongs to the ideal $\I$.

We claim that $\A$ is an algebra of sets in $X$. Indeed, for any set $A\in\A$ find an open set $U\in\tau$ with $A\symdiff U\in\I$ and observe that for the open set $V:=X\setminus\clo{U}$ we get $(X\setminus A)\symdiff V\subset (A\symdiff U)\cup (\clo{U}\setminus U)\in\I$, which implies that $X\setminus A\in\A$.

Given a family $\F\subset\A$ of cardinality $|\F|<\kappa$, for each set $F\in\F$ find an open set $U_F\subset X$ with $F\symdiff U_F\in\I$ and observe that for the open set $U:=\bigcup_{F\in\F}U_F$ we get $(\bigcup\F)\symdiff U\subset \bigcup_{F\in\F}(F\symdiff U_F)\in\I$. Therefore, $\A$ is a $\kappa$-additive algebra of sets in $X$, containing all open sets and all nowhere dense sets in $X$. Taking into account that $\kappa\mhyp\BP(X)$ is the smallest $\kappa$-additive algebra with this property, we conclude that $\kappa\mhyp\BP(X)\subset \A$. The reverse inclusion $\A\subset\kappa\mhyp\BP(X)$ is trivial.
\smallskip

Therefore, $\{x\}\in\kappa\mhyp\BP(X)=\A$ and hence $\{x\}\symdiff U\in\I$ for some open set $U\subset X$. If $x\notin U$, then $\{x\}\symdiff U=\{x\}\cup U\in\I$, which implies that $\{x\}\in \I$ and hence $\{x\}$ is nowhere dense in $X$. So, we assume that $x\in U$. In this case $\{x\}\symdiff U=U\setminus\{x\}\in\I$, so $U\setminus\{x\}=\bigcup\mathcal N$ for some family $\mathcal N$ of nowhere dense sets in $X$ of cardinality $|\mathcal N|<\kappa$. If for some set $N\in\mathcal N$ the closure $\clo{N}$ contains $x$, then the singleton $\{x\}\subset \clo{N}$ is nowhere dense in $X$. In the opposite case the singleton $\{x\}=U\setminus \bigcup_{N\in\mathcal N}\clo{N}$ is a $G_{<\kappa}$-set in $X$.
\end{proof}

It is an easy observation that an isolated point of a dense subset of a topological space is also isolated in the whole space.
This needs a weak separation hypothesis -- $T_1$ is enough.
It turns out that this property is equivalent to the axiom $T_{\w\mhyp\BP}$.

\begin{proposition}
	Let $X$ be a topological space. Every isolated point of every dense subset is isolated in $X$ if and only if $X$ is a $T_{\w\mhyp\BP}$-space.
	
	\begin{proof}
		First, suppose that $X$ is $T_{\w\mhyp\BP}$, $D \subset X$ is dense, and $x$ is an isolated point of $D$.
		There is an open set $U \subset X$ such that $U \cap D = \{x\}$.
		If $\{x\}$ is open in $X$, we are done. Otherwise, $\{x\}$ is nowhere dense in $X$, and so $U \not\subset \clo{\{x\}}$.
		But then $U \setminus \clo{\{x\}}$ is a nonempty open set disjoint with $D$, which is impossible.
		
		On the other hand, suppose that every isolated point of every dense subset is isolated in $X$, and let $x \in X$.
		We put $D := (X \setminus \clo{\{x\}}) \cup \{x\}$. Clearly, $D$ is dense.
		If $\{x\}$ is not nowhere dense in $X$, then there is a nonempty open set $U \subset \clo{\{x\}}$, which witnesses that $x$ is isolated in $D$. It follows from the assumption that $x$ is isolated in $X$.
	\end{proof}
\end{proposition}

\section{Implications and examples}

Because of the characterizations of the properties $T_{\kappa\mhyp\Borel}$ and $T_{\kappa\mhyp\BP}$ in the previous theorems, and because we have already considered the property $T_{\Closed\Or\Open}$, it makes sense to consider also the properties $T_\ClosedOrG{<\kappa}$ for uncountable regular cardinals $\kappa$ and $T_\ClosedOrG{\infty}$.
By Theorem~\ref{t:BP}, the axioms $T_{\omega\mhyp\BP}$ and $T_\OpenOrNwd$ are equivalent. We introduce also the axiom $T_\ClosedOrNwd$, which generalizes both $T_\Nwd$ and $T_\Closed$, and which is stronger than $T_\OpenOrNwd$ since every closed singleton that is not nowhere dense is necessarily open.
Finally, we denote the family of all regular open subsets of a topological space $X$ by $\RO(X)$, and we introduce the axioms $T_\ClosedOrRO$, $T_\ClosedMeetsRO$, and $T_\NwdOrRO$. These axioms naturally fit the diagram, and they are related to some known separation axioms. See Section~\ref{s:connections} for details.

All these additional properties make Diagram~\ref{diag_1} more complete. The resulting diagram with all implications follows ($\kappa$ is an arbitrary uncountable regular cardinal):
%\begin{equation} \label{diag_2}
%\xymatrix{
	%*+\txt{$T_1$, $T_\Closed$} \ar@{=>}[r] \ar@{=>}[dd] &
	%T_\ClosedOrRO \ar@{=>}[r] \ar@{=>}[d] &
	%*+\txt{$T_\ClosedOrOpen$, \\ $T_\ClosedOrG{<\w}$} \ar@{=>}[r] \ar@{=>}[d] &
	%*+\txt{$T_\ClosedOrG{\delta}$, \\ $T_\ClosedOrG{<\w_1}$} \ar@{=>}[r] \ar@{=>}[d] &
	%T_\ClosedOrG{<\kappa} \ar@{=>}[r] \ar@{=>}[d] &
	%*+\txt{$T_\ClosedOrG{\infty}$} \ar@{=>}[d] &
	%\\
	%&
	%T_\ClosedMeetsRO \ar@{=>}[r] \ar@{=>}[d] &
	%*+\txt{$T_\Constructible$, \\ $T_{\w\mhyp\Borel}$} \ar@{=>}[r] \ar@{=>}[d] &
	%*+\txt{$T_\Borel$, \\ $T_{\w_1\mhyp\Borel}$} \ar@{=>}[r] \ar@{=>}[d] &
	%T_{\kappa\mhyp\Borel} \ar@{=>}[r] \ar@{=>}[d] &
	%*+\txt{$T_0$, $T_{\infty\mhyp\Borel}$} \ar@{=>}[d] &
	%\\
	%*+\txt{$T_\ClosedOrNwd$} \ar@{=>}[r] &
	%T_\NwdOrRO \ar@{=>}[r] &
	%*+\txt{$T_\OpenOrNwd$, \\ $T_{\w\mhyp\BP}$} \ar@{=>}[r] &
	%*+\txt{$T_\BP$, $T_{\w_1\mhyp\BP}$} \ar@{=>}[r] &
	%T_{\kappa\mhyp\BP} \ar@{=>}[r] &
	%*+\txt{$T_{\infty\mhyp\BP}$} &
%}
%\end{equation}
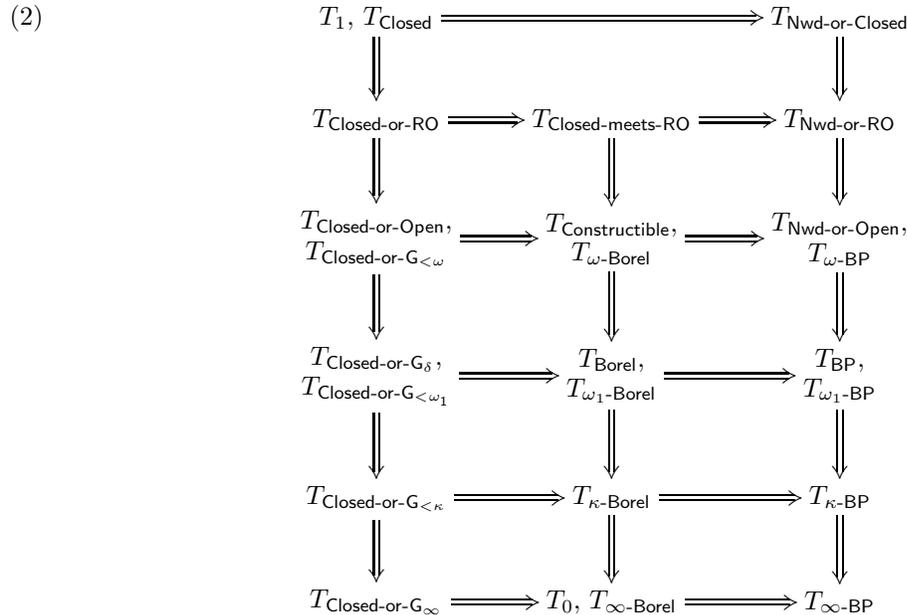
\begin{equation} \label{diag_2}
\xymatrix{
	*+\txt{$T_1$, $T_\Closed$} \ar@{=>}[rr] \ar@{=>}[d] &
	&
	*+\txt{$T_\ClosedOrNwd$} \ar@{=>}[d] &
	\\
	T_\ClosedOrRO \ar@{=>}[r] \ar@{=>}[d] &
	T_\ClosedMeetsRO \ar@{=>}[r] \ar@{=>}[d] &
	T_\NwdOrRO \ar@{=>}[d] &
	\\
	*+\txt{$T_\ClosedOrOpen$, \\ $T_\ClosedOrG{<\w}$} \ar@{=>}[r] \ar@{=>}[d] &
	*+\txt{$T_\Constructible$, \\ $T_{\w\mhyp\Borel}$} \ar@{=>}[r] \ar@{=>}[d] &
	*+\txt{$T_\OpenOrNwd$, \\ $T_{\w\mhyp\BP}$} \ar@{=>}[d] &
	\\
	*+\txt{$T_\ClosedOrG{\delta}$, \\ $T_\ClosedOrG{<\w_1}$} \ar@{=>}[r] \ar@{=>}[d] &
	*+\txt{$T_\Borel$, \\ $T_{\w_1\mhyp\Borel}$} \ar@{=>}[r] \ar@{=>}[d] &
	*+\txt{$T_\BP$, \\ $T_{\w_1\mhyp\BP}$} \ar@{=>}[d] &
	\\
	T_\ClosedOrG{<\kappa} \ar@{=>}[r] \ar@{=>}[d] &
	T_{\kappa\mhyp\Borel} \ar@{=>}[r] \ar@{=>}[d] &
	T_{\kappa\mhyp\BP} \ar@{=>}[d] &
	\\
	*+\txt{$T_\ClosedOrG{\infty}$} \ar@{=>}[r] &
	*+\txt{$T_0$, $T_{\infty\mhyp\Borel}$} \ar@{=>}[r] &
	*+\txt{$T_{\infty\mhyp\BP}$} &
}
\end{equation}
Note that the \tr{rows}{columns} of the diagram correspond to the axiom patterns $T_{\Closed\Or\bullet}$, $T_{\Closed\Meets\bullet}$, and $T_{\Nwd\Or\bullet}$, while the \tr{columns}{rows} correspond to $\bullet$ ranging $\Closed$ (or equivalently $\Clopen$), $\RO$, $\Open$, $\G{\delta}$, $\G{<\kappa}$, and $\G{\infty}$.

The implications are clear or easily follow from what has already been proved. We comment just on $T_\ClosedMeetsRO \implies T_\NwdOrRO$. If we have $\{x\} = F \cap U$ where $F$ is closed and $U$ is regular open and the singleton $\{x\}$ is not nowhere dense, then we also have $\{x\} = \int(\clo{\{x\}}) \cap U$, and this is a regular open set.
\medskip

A topological space $X$ is {\em symmetric} if for any points $x,y\in X$ the
existence of an open set $U_x\subset X$ containing $x$ but not $y$ is
equivalent to the existence of an open set $U_y\subset X$ containing $y$ but
not $x$.
It follows that a topological space is $T_1$ if and only if is it $T_0$ and symmetric.
Similarly, it can be proved that a topological space is $T_\ClosedOrNwd$ if it is $T_{\infty\mhyp\BP}$ and symmetric.
Hence, for symmetric spaces Diagram~\ref{diag_2} collapses \tr{horizontally}{vertically}.

A topological space $X$ is {\em subfit} if for every open sets $U, V \subset X$ such that $U \not\subset V$ there is an open set $W \subset X$ such that $U \cup W = X \neq V \cup W$. This property clearly depends only on the lattice of open sets, and is in fact a separation axiom considered in point-free topology \cite[V.1]{Picado_Pultr_12} -- it is a weaker but point-free variant of $T_1$. A topological space is subfit if and only if for every point $x \in X$ and every its neighborhood $U$ there is $y \in \clo{\{x\}}$ such that $\clo{\{y\}} \subset U$.

We are interested in the subfit condition since a topological space is $T_1$ if and only if it is $T_\Constructible$ and subfit \cite[V.1.1]{Picado_Pultr_12}. Similarly, it can be proved that a topological space is $T_\ClosedOrNwd$ if it is $T_{\w\mhyp\BP}$ and subfit.
Hence, for subfit spaces we have a \tr{horizontal}{vertical} collapse of Diagram~\ref{diag_2} just for the \tr{left}{top} part.
This suggests the following definition.

\begin{definition}
	Let $\kappa$ be an infinite cardinal. We say that a topological space is $\kappa$-subfit if for every $\G{<\kappa}$-set $U \subset X$ and an open set $V \subset X$ such that $U \not\subset V$ there is an open set $W \subset X$ such that $U \cup W = X \neq V \cup W$. Equivalently, if for every $x \in X$ and every $\G{<\kappa}$-set $U \ni x$ there is $y \in \clo{\{x\}}$ such that $\clo{\{y\}} \subset U$.
	Analogously, we define $\RO$-subfit spaces and $\infty$-subfit spaces -- the set $U$ is supposed to be regular open or $\G{\infty}$ rather than $\G{<\kappa}$.
	
	This way we obtain a sequence of subfitness conditions, one for each \tr{column}{row} of Diagram~\ref{diag_2}. The subfitness condition for the first \tr{column}{row}, where the set $U$ would be closed or equivalently clopen, is trivial.
	Clearly, a topological space is subfit if and only if it is $\w$-subfit.
	Also, we have the implications $\infty$-subfit $\implies$ $\kappa$-subfit $\implies$ $\w$-subfit $\implies$ $\RO$-subfit.
\end{definition}

\begin{proposition} \label{t:symmetric_subfit}
	A topological space $X$ is symmetric if and only if it is $\infty$-subfit if and only if it is hereditarily subfit.
	
	\begin{proof}
		It is easy to see that $X$ is symmetric if for every $x \in X$ and every its neighborhood $U$ we have $\clo{\{x\}} \subset U$. It follows that $\clo{\{x\}} \subset U_x := \bigcap\{U \subset X$ open $: x \in U\}$, which is the smallest $\G{\infty}$-subset containing $x$, and hence $X$ is $\infty$-subfit and in particular subfit. Since being symmetric is clearly a hereditary property, we also have that $X$ is hereditarily subfit.
		
		On the other hand, suppose that $X$ is $\infty$-subfit. Let $x \in X$ and let $U_x$ be as above. By $\infty$-subfitness there is $y \in \clo{\{x\}}$ such that $\clo{\{y\}} \subset U_x$. We have $x \in \clo{\{y\}}$ since otherwise $U_x \setminus \clo{\{y\}}$ would be a $\G{\infty}$-set containing $x$ strictly smaller than $U_x$. We have $\clo{\{x\}} \subset U_x$, and $X$ is symmetric.
		
		Finally, suppose that $X$ is hereditarily subfit. Let $x \in X$ and let $U$ be an open neighborhood of $x$. We consider the subspace $Y := \{x\} \cup (X \setminus U)$. We have that $x$ is an isolated point of a subfit space $Y$, and so $\{x\}$ is closed in $Y$. This is because from subfitness we have $y \in \{x\}$ such that $\clo{\{y\}} \cap Y \subset \{x\}$. It follows that $\clo{\{x\}} \subset U$, and $X$ is symmetric.
	\end{proof}
\end{proposition}

Now we prove a general proposition on \tr{horizontal}{vertical} collapses of the diagram.

\begin{proposition} \label{t:kappa_subfit}
	Let $X$ be a topological space and let $\kappa$ be an infinite regular cardinal.
	\begin{enumerate}
		\item $X$ is $T_1$ if and only if it is $T_{\kappa\mhyp\Borel}$ and $\kappa$-subfit, if and only if it is $T_{\infty\mhyp\Borel}$ and $\infty$-subfit, if and only if it is $T_\ClosedMeetsRO$ and $\RO$-subfit.
		\item $X$ is $T_\ClosedOrNwd$ if it is $T_{\kappa\mhyp\BP}$ and $\kappa$-subfit, or $T_{\infty\mhyp\BP}$ and $\infty$-subfit, or $T_{\NwdOrRO}$ and $\RO$-subfit.
	\end{enumerate}
	
	\begin{proof}
		Clearly, if $X$ is $T_1$, then it has all the other properties.
		On the other hand, suppose that $\{x\} = F \cap U$ for a closed set $F \subset X$ and a $\G{<\kappa}$-set $U \subset X$. By $\kappa$-subfitness there is a point $y \in \clo{\{x\}} \subset F$ such that $\clo{\{y\}} \subset U$. We have $\clo{\{y\}} \subset F \cap U = \{x\}$, and hence $y = x$ and $\clo{\{x\}} = \{x\}$.
		It follows that a $T_{\kappa\mhyp\Borel}$ and $\kappa$-subfit space is $T_1$.
		
		Finally, suppose that $X$ is $T_{\kappa\mhyp\BP}$ and $\kappa$-subfit. Let $x \in X$.
		If $\{x\}$ is nowhere dense, we are done. If $\{x\}$ is a $\G{<\kappa}$-set, then by the same argument as above, $\{x\}$ is closed since $X$ is $\kappa$-subfit. Together, $X$ is $T_\ClosedOrNwd$.
		
		The proofs for the $\infty$ and $\RO$ cases are analogous.
	\end{proof}
\end{proposition}

\begin{corollary}
	A topological space $X$ is $T_1$ if and only if every its subspace is $T_{\infty\mhyp\BP}$ and subfit.
	
	\begin{proof}
		If $X$ is hereditarily $T_{\infty\mhyp\BP}$, then by Proposition~\ref{p:closed_hereditarily_Borel}, which we prove later, $X$ is $T_0$. If $X$ is hereditarily subfit, then by Proposition~\ref{t:symmetric_subfit} it is symmetric. Together, $X$ is $T_1$.
		The other implication is clear.
	\end{proof}
\end{corollary}

\begin{remark}
	Later in this section we will distinguish between the subfitness axioms (Example~\ref{e:kappa} and \ref{e:Sierpinski}).
	Also note that the other implication in Proposition~\ref{t:kappa_subfit} (2) does not hold, $T_\ClosedOrNwd$ does not imply even $\RO$-subfitness (Example~\ref{e:Khalimsky}).
\end{remark}

\begin{remark}
	While the subfitness is a point-free condition, i.e. it really depends only on the lattice of open sets, this is not the case with $\kappa$-subfitness and $\infty$-subfitness. The space from Example~\ref{e:kappa} is not $\kappa^+$-subfit. On the other hand, $\kappa$ with the cofinite topology has isomorphic lattice of open sets, but is even $T_1$.
	
	Being $\RO$-subfit is a point-free condition since $U \subset X$ is regular open if and only if $U = U^{**}$ where $U^* := X \setminus \clo{U}$ is the pseudocomplement of $U$ in the lattice of open sets.
	
	There is also a point-free condition of being \emph{fit} \cite[V.1.2]{Picado_Pultr_12}. For topological spaces, this condition lies strictly between regularity and symmetry.
\end{remark}

Now we shall consider the \tr{vertical}{horizontal} nature of Diagram~\ref{diag_2}.
Recall that a topological space is \emph{nodec} if every its nowhere dense subset is closed.

\begin{proposition}
	Diagram~\ref{diag_2} collapses \tr{vertically}{horizontally} for nodec spaces. More precisely, let $X$ be a nodec space and let $\kappa$ be an infinite regular cardinal.
	\begin{enumerate}
		\item $X$ is $T_\ClosedOrNwd$ if and only if $X$ is $T_1$;
		\item $X$ is $T_\NwdOrRO$ if and only if $X$ is $T_\ClosedOrRO$;
		\item $X$ is $T_{\kappa\mhyp\BP}$ if and only if $X$ is $T_{\kappa\mhyp\Borel}$ if and only if $X$ is $T_\ClosedOrG{<\kappa}$;
		\item $X$ is $T_{\infty\mhyp\BP}$ is and only if $X$ is $T_{\infty\mhyp\Borel}$ if and only if $X$ is $T_\ClosedOrG{\infty}$.
	\end{enumerate} 	
	\begin{proof}
		Is trivial or follows easily from the characterization in Theorem~\ref{t:BP}.
	\end{proof}
\end{proposition}

Let us recall several kinds of nearly open sets. A subset $A$ of a topological space $X$ is called
\begin{itemize}
	\item \emph{semi-open} \cite{semi-open} if there is an open set $U$ such that $U \subset A \subset \clo{U}$, or equivalently if $A \subset \clo{\int(A)}$;
	\item \emph{pre-open} \cite{pre-open} if there is an open set $U$ such that $A \subset U \subset \clo{A}$, or equivalently if $A \subset \int(\clo{A})$;
	\item \emph{$\alpha$-open} \cite{alpha-open} if there are open sets $U$, $V$ such that $U \subset A \subset V \subset \clo{U}$, or equivalently if $A \subset \int(\clo{\int(A)})$;
	\item \emph{$\beta$-open} \cite{beta-open} or \emph{semi-preopen} \cite{semi-preopen} if there is an open set $U$ such that $A \subset \clo{U}$ and $U \subset \clo{A}$, or equivalently if $A \subset \clo{\int(\clo{A})}$.
\end{itemize}
Note that the notion of $\alpha$-open set it the strongest and the notion of $\beta$-open set is the weakest. Also, every open set is $\alpha$-open, $\alpha$-open sets are exactly sets that are both semi-open and pre-open, and pre-open sets are exactly intersections of open and dense sets.

Let $(X, \tau)$ be a topological space. Nj\r{a}stad showed in \cite{alpha-open} that the family of all $\alpha$-open sets forms a topology, denoted by $\tau^\alpha$. Moreover, the $\alpha$-open sets are exactly the sets of form $U \setminus N$ where $U$ is open and $N$ is nowhere dense, so $\tau^\alpha$ is generated by $\tau$ and by declaring the nowhere dense sets closed. Also, by doing this we do not introduce any new non-closed nowhere dense subsets, so the space $(X, \tau^\alpha)$ is nodec. Moreover, $(X, \tau)$ and $(X, \tau^\alpha)$ have the same regular open sets. This was proved in \cite[Propositions~3 and 6]{alpha-open}, even though Nj\r{a}stad used a different terminology.

\begin{proposition} \label{t:alpha}
	Let $(X, \tau)$ be a topological space, let $\tau^\alpha$ be the corresponding nodec modification, and let $\kappa$ be an infinite regular cardinal.
	\begin{enumerate}
		\item $(X, \tau)$ is $T_\ClosedOrNwd$ if and only if $(X, \tau^\alpha)$ is $T_1$;
		\item $(X, \tau)$ is $T_\NwdOrRO$ if and only if $(X, \tau^\alpha)$ is $T_\ClosedMeetsRO$ or equivalently $T_\ClosedOrRO$;
		\item $(X, \tau)$ is $T_{\kappa\mhyp\BP}$ if and only if $(X, \tau^\alpha)$ is $T_{\kappa\mhyp\Borel}$ or equivalently $T_\ClosedOrG{<\kappa}$;
		\item $(X, \tau)$ is $T_{\infty\mhyp\BP}$ is and only if $(X, \tau^\alpha)$ is $T_{\infty\mhyp\Borel}$ or equivalently $T_\ClosedOrG{\infty}$.
	\end{enumerate}
 	
	\begin{proof}
		First, observe that since every $\alpha$-closed set is the union of a closed set and a nowhere dense set, any singleton is $\alpha$-closed if and only if it is closed or nowhere dense or equivalently clopen or nowhere dense. This gives us the first equivalence.
		
		The second equivalence follows from the previous observation and from the fact that a singleton $\{x\}$ is regular open in $\tau$ if and only if it is regular open in $\tau^\alpha$.
		
		For the third equivalence we use the characterization in Theorem~\ref{t:BP}. If $\{x\}$ is nowhere dense or $G_{<\kappa}$, then it is clearly closed or $G_{<\kappa}$ in $\tau^\alpha$. For the other implication let $\{x\}$ be closed or $G_{<\kappa}$ in $\tau^\alpha$. If it is closed in $\tau^\alpha$, we are done by the first observation.
		In the second case we have that $\{x\} = \bigcap_{\beta < \kappa} V_\beta$ where every set $V_\beta$ is open in $\tau^\alpha$, i.e.\ there is a set $U_\beta$ that is open in $\tau$ and a set $N_\beta$ that is nowhere dense in $\tau$ such that $V_\beta = U_\beta \setminus N_\beta$. We may suppose that the singleton $\{x\}$ is not nowhere dense, so we have $x \notin \clo{N_\beta}$. Hence, $\{x\} = \bigcap_{\beta < \kappa} U_\beta \setminus \clo{N_\beta}$, which is a $G_{<\kappa}$-set.
		
		The fourth equivalence follows from the third one.
	\end{proof}
\end{proposition}

\begin{proposition} \label{t:meets}
	Let $X$ be a topological space.
	\begin{enumerate}
		\item $X$ is $T_1$ if and only if $X$ is $T_\ClosedOrNwd$ and $T_\ClosedOrOpen$;
		\item $X$ is $T_\ClosedOrRO$ if and only if $X$ is $T_\NwdOrRO$ and $T_\ClosedOrOpen$.
	\end{enumerate}
	\begin{proof}
		The ``only if'' part is trivial. For the ``if'' part let us assume that $X$ is $T_\ClosedOrNwd$ and $T_\ClosedOrOpen$. Then every singleton is either closed or both nowhere dense and open, but the latter case is contradictory. Similarly, if $X$ is $T_\NwdOrRO$ and $T_\ClosedOrOpen$, then every singleton is closed or regular open or both nowhere dense and open.
	\end{proof}
\end{proposition}

Next, we shall present some examples distinguishing the separation axioms
$T_\A$ for various algebras $\A$. First we observe that the separation axioms
in the Diagram~\ref{diag_2} are not trivial (i.e., fail for some topological
spaces).

\begin{example} The doubleton $A=\{0,1\}$ endowed with the anti-discrete topology $\{\emptyset,A\}$ fails to be a $T_{\infty\mhyp\BP}$-space.
\end{example}

\begin{example} Let $A=\{0,1\}$ be a doubleton endowed with the anti-discrete topology $\{\emptyset,A\}$, and $[0,1]$ be unit interval, endowed with the standard Euclidean topology. It is clear that $A\times[0,1]$ is not a $T_0$-space and hence fails to be a $T_{\infty\mhyp\Borel}$-space. On the other hand, the product $A\times [0,1]$ is a $T_\Nwd$-space.
\end{example}

\begin{example} \label{e:kappa}
Let $\kappa$ be a regular infinite cardinal. On the space $X=\kappa\cup\{\kappa\}$ consider the topology $$\tau=\{\emptyset\}\cup\{U\subset X:\kappa\in U \wedge |X\setminus U|<\w\}.$$ The topological space $(X,\tau)$ has the following properties:
\begin{enumerate}
\item $X$ is a compact $T_0$-space;
\item $X$ is a $T_\ClosedOrG{<\kappa^+}$-space and hence a $T_{\kappa^+\mhyp\Borel}$-space;
\item $X$ fails to be a $T_{\kappa\mhyp\BP}$-space;
\item $X$ is $\kappa$-subfit but not $\kappa^+$-subfit.
\item $X\times[0,1]$ is a compact $T_\Nwd$-space which is still a $T_\ClosedOrG{<\kappa^+}$-space but not a $T_{\kappa\mhyp\Borel}$-space.
\end{enumerate}
\end{example}

\begin{example} \label{e:Sierpinski}
	Recall that every partially ordered set $(P, \leq)$ induces the corresponding \emph{Alexandrov topology}. Open sets are precisely the upper sets, i.e., the sets $U \subset P$ such that $x \in U$ and $x \leq y$ implies $y \in U$.
	\begin{enumerate}
		\item The two-point Sierpi\'nski space $S_2 = \{0, 1\}$ with the Alexandrov topology $\{\{0, 1\}, \{1\}, \emptyset\}$ is $T_\ClosedOrOpen$ but not $T_\NwdOrRO$.
		Also, it is $\RO$-subfit but not subfit.
		\item The three-point analogue of the Sierpi\'nski space $S_3 =
\{0, 1, 2\}$ with the Alexandrov topology $\{\{0, 1, 2\},$ $ \{1, 2\},
\{2\}, \emptyset\}$ is $T_\Constructible$ but is neither
$T_\ClosedOrG{\infty}$ nor $T_\NwdOrRO$.
		\item The $\w$-analogue, i.e., $S_\w = \w$ with the Alexandrov
topology is $T_\Constructible$ and $T_{\Nwd}$ but not
$T_\ClosedOrG{\infty}$ and not $T_\ClosedMeetsRO$.
	\end{enumerate}
\end{example}

\begin{example} \label{e:Khalimsky}
	Recall that the set of all integers $\mathbb{Z}$ endowed with the topology generated by the sets $\{2k - 1,\, 2k,\, 2k + 1\}$ for $k \in \mathbb{Z}$ is called the \emph{Khalimsky line} or the \emph{digital line}.
	\begin{enumerate}
		\item The Khalimsky line is $T_\ClosedOrRO$ but not $T_\ClosedOrNwd$. Every odd singleton is regular open, while every even singleton is closed and nowhere dense.
		\item The subspace $\{-1, 0, 1\}$ of the Khalimsky line, which may be viewed as one open segment of the line, has the same properties, i.e.\ it is $T_\ClosedOrRO$ but not $T_\ClosedOrNwd$.
		\item The space $\{-1, 0, 1\} \times \mathbb{R}$ is $T_\ClosedMeetsRO$ and $T_\Nwd$, but not $T_1$ and so not $T_\ClosedOrNwd$ (\ref{t:meets}) and not $\RO$-subfit (\ref{t:kappa_subfit}).
		\item Let us consider the following attachment of two copies of the previous space: the set $\{-1, 0, 1_{-1}, 1_0, 1_1\}$ endowed with the topology generated by the sets $\{-1\}$, $\{1_{-1}, 1_0, 1_1\}$, $\{1_{-1}\}$, $\{1_1\}$. This space is $T_\ClosedMeetsRO$ but neither $T_\ClosedOrNwd$ nor $T_\ClosedOrG{\infty}$. The singletons $\{-1\}$, $\{1_{-1}\}$, $\{1_1\}$ are regular open, the singleton $\{0\}$ is closed, and the singleton $\{1_0\}$ is the intersection of the closed set $\{0, 1_0\}$ and the regular open set $\{1_{-1}, 1_0, 1_1\}$.
	\end{enumerate}
\end{example}

\section{Preservation properties}

Finally, we establish some hereditary properties of the separation axioms $T_\A$. The following is obvious.

\begin{proposition} Let $\kappa$ be an infinite cardinal. Any subspace of a $T_\ClosedOrG{<\kappa}$-space is a $T_\ClosedOrG{<\kappa}$-space. Any subspace of a $T_\ClosedOrG{\infty}$-space is a $T_\ClosedOrG{\infty}$-space.
\end{proposition}

Theorem~\ref{t:Borel} and Remark~\ref{t:regular} imply the following
proposition.

\begin{proposition}\label{p1} Let $\kappa$ be an infinite cardinal. Any subspace of a $T_{\kappa\mhyp\Borel}$-space is a $T_{\kappa\mhyp\Borel}$-space.
\end{proposition}

\begin{proposition} \label{p:closed_hereditarily_Borel}
	Let $\kappa$ be an infinite cardinal. A topological space $X$ is a $T_{\kappa\mhyp\Borel}$-space if and only if each closed subspace of $X$ is a $T_{\kappa\mhyp\BP}$-space.
	Also, $X$ is $T_{\infty\mhyp\Borel}$ (i.e. $T_0$) if and only if each closed subspace of $X$ is a $T_{\infty\mhyp\BP}$-space.
\end{proposition}

\begin{proof} The ``only if'' part follows from Proposition~\ref{p1}. To prove the ``only if'' part, assume that each closed subspace of $X$ is a $T_{\kappa\mhyp\BP}$-space and that $\kappa$ is regular. Given any point $x\in X$ consider the closure $\clo{\{x\}}$ of the singleton $\{x\}$. Taking into account that $\{x\}$ is dense in $\clo{\{x\}}$ and $\clo{\{x\}}$ is a $T_{\kappa\mhyp\BP}$-space, we can apply Theorem~\ref{t:BP} and conclude that the singleton $\{x\}$ is a $G_{<\kappa}$-set in $\clo{\{x\}}$ and consequently, belongs to the algebra $\kappa\mhyp\Borel(X)$.

The $\infty$ case follows since $\infty\mhyp\Borel(X) = |X|^+\mhyp\Borel(X)$ and $\infty\mhyp\BP(Y) = |Y|^+\mhyp\BP(Y) = |X|^+\mhyp\BP(Y)$ for every $Y \subset X$.
\end{proof}

\begin{proposition}
	A topological space $X$ is $T_1$ if and only if each closed subspace of $X$ is $T_\NwdOrRO$.
	
	\begin{proof}
		The necessity is clear. For the sufficiency let $x \in X$. The singleton $\{x\}$ is dense in $\clo{\{x\}}$, so it is not nowhere dense in $\clo{\{x\}}$, and it is regular open in $\clo{\{x\}}$ only if $\clo{\{x\}} = \{x\}$.
	\end{proof}
\end{proposition}

To summarize, in Diagram~\ref{diag_2} exactly the axioms in the first two \tr{rows}{columns} with the exception of $T_\ClosedOrRO$ and $T_\ClosedMeetsRO$ are hereditary.
Also, the hereditary variant of a non-hereditary axiom is the weakest stronger hereditary axiom in the diagram, and for this, only satisfying the axiom closed-hereditarily is enough.

Let us investigate the hereditary properties of the non-hereditary axioms.
The class of $T_{\kappa\mhyp\BP}$-spaces is not hereditary because a nowhere dense set does not have to be nowhere dense in every subspace containing it.

\begin{proposition} \label{t:nwd}
	Let $X$ be a topological space and $A \subset X$. Every subset of $A$ that is nowhere dense in $X$ is nowhere dense in $A$ if and only if $A$ is $\beta$-open.
	
	\begin{proof}
		The condition that $A$ is $\beta$-open is necessary since $A \setminus \clo{\int(\clo{A})}$ is open in $A$ but nowhere dense in $X$. Now we show that the condition is also sufficient. Suppose that $A$ is $\beta$-open. If $N \subset A$ is not nowhere dense in $A$, then there is a an open set $U$ such that $\emptyset \neq U \cap A \subset \clo{N} \cap A$. Let us consider the set $V = U \cap \int(\clo{A})$. Clearly, $V$ is open and we have $V \subset U \cap \clo{A} \subset \clo{U \cap A} \subset \clo{N}$. Finally, we have $\emptyset \neq U \cap A \subset U \cap \clo{\int(\clo{A})} \subset \clo{U \cap \int(\clo{A})} = \clo{V}$. Therefore, $V \neq \emptyset$ and $N$ is not nowhere dense in $X$.
	\end{proof}
\end{proposition}

For every family of sets $\A$ and a set $B$ we denote the family $\{A \cap B: A \in \A\}$ by $\A\restr{B}$.

\begin{proposition} \label{t:BP-restriction}
	Let $\kappa$ be an infinite cardinal, let $X$ be a topological space and let $U \subset X$.
	\begin{itemize}
		\item If $U$ is $\beta$-open, then $\kappa\mhyp\BP(U) = \kappa\mhyp\BP(X)\restr{U}$.
		\item If $U$ is semi-open, then $U \in \kappa\mhyp\BP(X)$ and $\kappa\mhyp\BP(U) = \kappa\mhyp\BP(X) \cap \mathcal P(U)$.
	\end{itemize}
	
	\begin{proof}
		For the first part, let us consider the map $f: \mathcal P(X) \to \mathcal P(U)$ defined by $f(A) = A \cap U$ for every $A \subset X$. The map $f$ preserves arbitrary unions, intersections, and complements. Hence, $f(\A)$ is a $\kappa$-additive subalgebra of $\mathcal P(U)$ for every $\kappa$-additive subalgebra $\A \subset \mathcal P(X)$, and also $f^{-1}(\A)$ is a $\kappa$-additive subalgebra of $\mathcal P(X)$ for every $\kappa$-additive subalgebra $\A \subset \mathcal P(U)$. Moreover, if $\A$ is the smallest $\kappa$-additive subalgebra of $\mathcal P(X)$ containing some family $\F$, then $f(\A)$ is the smallest $\kappa$-additive subalgebra of $\mathcal P(U)$ containing the family $f(\F)$. In our case, $\kappa\mhyp\BP(X)\restr{U}$ is the smallest $\kappa$-additive subalgebra of $\mathcal P(U)$ containing $\Open(X)\restr{U} \cup \Nwd(X)\restr{U}$. Clearly, $\Open(X)\restr{U} = \Open(U)$ and by Proposition~\ref{t:nwd} we have also $\Nwd(X)\restr{U} = \Nwd(U)$, which concludes the proof.
		
		For the second part, since $U$ is semi-open, there is an open set $V$ such that $V \subset U \subset \clo{V}$, so $U$ is the union of an open set and a nowhere dense set and hence a member of $\kappa\mhyp\BP(U)$. For every $\kappa$-additive subalgebra $\A \subset \mathcal P(X)$ we have $\A \cap \mathcal P(U) \subset \A\restr{U}$. On the other hand, $\A\restr{U} \subset \A \cap \mathcal P(U)$ if and only if $U \in \A$, which is our case. Hence, by also using the first part we have $\kappa\mhyp\BP(U) = \kappa\mhyp\BP(X)\restr{U} = \kappa\mhyp\BP(X) \cap \mathcal P(U)$.
	\end{proof}
\end{proposition}

The following is obvious by using Proposition~\ref{t:nwd}.

\begin{corollary}
	Any $\beta$-open subspace of a $T_\Nwd$-space is a $T_\Nwd$-space. Any $\beta$-open subspace of a $T_\ClosedOrNwd$-space is a $T_\ClosedOrNwd$-space.
\end{corollary}

Either from Proposition~\ref{t:BP-restriction} or from Proposition~\ref{t:nwd} by using Remark~\ref{t:regular} and the characterization in Theorem~\ref{t:BP} we obtain the following.

\begin{corollary} Let $\kappa$ be an infinite cardinal. Any $\beta$-open subspace of a $T_{\kappa\mhyp\BP}$-space is a $T_{\kappa\mhyp\BP}$-space. The same holds for $T_{\infty\mhyp\BP}$-spaces.
\end{corollary}

\begin{example} The space $A = \{0, 1, 2\}$ endowed with the topology $\{\emptyset, \{0\}, A\}$ is a $T_{\w\mhyp\BP}$-space, but its closed nowhere dense subspace $B = \{1, 2\}$ is not even a $T_{\infty\mhyp\BP}$-space. Hence, not every subspace having the Baire property can be allowed in the previous corollary.
\end{example}

In order to establish the hereditary properties of the axioms based on regular open sets we first need to look at preservation of regular open sets. Variants of the following proposition are known.

\begin{proposition}
	Let $X$ be a topological space. If $A$ is a pre-open subset of $X$, then $\RO(A) = \RO(X)\restr{A}$.
	\begin{proof}
		We will show that for every open $U \subset X$ we have $\ro_A(U \cap A) = \ro(U) \cap A$, where $\ro$ is a shortcut for the interior of the closure. If $V \subset X$ is open, then we have $\ro_V(U \cap V) = \int_V(\clo{U \cap V} \cap V) = \int(\clo{U} \cap V) = \ro(U) \cap V$, so the claim holds for $A$ open. Let $D \subset X$ be dense. It generally holds that $\int_D(F \cap D) = \int(F) \cap D$ for every closed $F \subset X$. By considering $F = \clo{U \cap D} = \clo{U}$ we obtain $\ro_D(U \cap D) = \int_D(\clo{U \cap D} \cap D) = \ro(U) \cap D$, so the claim holds also for $A$ dense. As a pre-open set, $A$ is dense in some open set $V \subset X$. We have $\ro_A(U \cap A) = \ro_A((U \cap V) \cap A) = \ro_V(U \cap V) \cap A = \ro(U) \cap V \cap A = \ro(U) \cap A$.
	\end{proof}
\end{proposition}

The previous proposition cannot be generalized to $\beta$-open subsets $A$. It is not true even for a regular closed set $A$ since $\int(A)$ would be regular open in the whole space but dense in $A$.

\begin{corollary}
	Any pre-open subspace of a $T_\ClosedOrRO$-space is a $T_\ClosedOrRO$-space. The same holds for $T_\ClosedMeetsRO$ and $T_\NwdOrRO$.
\end{corollary}

The following example shows that the previous corollary cannot be generalized to $\beta$-open subspaces.

\begin{example}
	The Khalimsky line is a $T_\ClosedOrRO$-space (Example~\ref{e:Khalimsky}), but its regular closed (and hence semi-open and $\beta$-open) subspace $\{0, 1, 2\}$ is not even $T_\NwdOrRO$.
\end{example}

Clearly, the subset of all non-isolated points of a $T_\ClosedOrOpen$-space is $T_1$, so in particular the property $T_\ClosedOrRO$ is hereditary with respect to meager subsets (but not to all subsets with the Baire property as the previous example shows). On the other hand, the following example shows that the properties $T_\ClosedMeetsRO$ and $T_\NwdOrRO$ are not hereditary even to closed nowhere dense subsets.

\begin{example}
	The space $\{-1, 0, 1_{-1}, 1_0, 1_1\}$ from Example~\ref{e:Khalimsky} is $T_\ClosedMeetsRO$, but its closed nowhere dense subspace $\{0, 1_0\}$ (homeomorphic to the Sierpi\'nski space) is not even $T_\NwdOrRO$.
\end{example}

Next, let us consider preservation of the separation axioms under products.

\begin{example}
	The Khalimsky line $K$ (Example~\ref{e:Khalimsky}) is $T_\ClosedOrRO$, but the square $K \times K$ is not even $T_\ClosedOrG{\infty}$. Therefore, the properties $T_\ClosedOrRO$ and $T_\ClosedOrG{<\kappa}$ are easily destroyed by products.
	
	%The two-point Sierpi\'nski space $S_2$ (Example~\ref{e:Sierpinski}) is
%$T_{\ClosedOrOpen}$, but the square $S_2 \times S_2$ is not even a
%$T_\ClosedOrG{\infty}$-space. Therefore, the properties of being a
%$T_\ClosedOrG{<\kappa}$ are easily destroyed by products.
\end{example}

For an infinite cardinal $\kappa$ let $\kappa^*$ denote $\kappa$ if $\kappa$ is regular, and $\kappa^+$ if it is singular.
The following proposition can be easily derived from Theorem~\ref{t:Borel} and for singular $\kappa$ from Remark~\ref{t:regular}.

\begin{proposition} For any infinite cardinal $\kappa$, any set $I$ of cardinality $|I|<\kappa^*$ and any family $\{X_i\}_{i\in I}$ of $T_{\kappa\mhyp\Borel}$-spaces the Tychonoff product $\prod_{i\in I}X_i$ is a $T_{\kappa\mhyp\Borel}$-space.
\end{proposition}

\begin{remark} \label{t:nwd_in_product}
	Let $\{(x_i)_{i \in I}\}$ be a singleton in a Tychonoff product
$\prod_{i\in I} X_i$. Its closure $\prod_{i\in I} \clo{\{x_i\}}$ contains an
open subset if and only if there exists a finite set $F\subset I$ such that for
every $i\in F$ the singleton $\{x_i\}$ is not nowhere dense in $X_i$, and for
every $i\in I\setminus F$ the singleton $\{x_i\}$ is dense in $X_i$.
 Hence, the singleton $\{(x_i)_{i \in I}\}$ is nowhere dense
if and only if $\{x_i\}$ is nowhere dense in $X_i$ for some $i\in I$ or there
are infinitely many indices $i\in I$ such that $\{x_i\}$ is not dense in $X_i$.
\end{remark}

The following proposition can be easily derived from Theorem~\ref{t:BP}, Remark~\ref{t:nwd_in_product} and for singular $\kappa$ from Remark~\ref{t:regular}.

\begin{proposition} For any infinite cardinal $\kappa$, any set $I$ of cardinality $|I|<\kappa^*$ and any family $\{X_i\}_{i\in I}$ of $T_{\kappa\mhyp\BP}$-spaces the Tychonoff product $\prod_{i\in I}X_i$ is a $T_{\kappa\mhyp\BP}$-space.
\end{proposition}

\begin{corollary} For any family $\{X_i\}_{i\in I}$ of $T_{\infty\mhyp\BP}$-spaces the Tychonoff product $\prod_{i\in I}X_i$ is a $T_{\infty\mhyp\BP}$-space.
\end{corollary}

From Remark~\ref{t:nwd_in_product} we also obtain the following.

\begin{proposition}
	For any $T_\Nwd$-space $X$ and any topological space $Y$ the product $X \times Y$ is a $T_\Nwd$-space.
\end{proposition}
	
\begin{proposition}
	For any family $\{X_i\}_{i\in I}$ of\/ $T_\ClosedOrNwd$-spaces the product
$\prod_{i\in I}X_i$ is a $T_\ClosedOrNwd$-space. Moreover, if infinitely many
of the spaces $X_i$ are nondegenerate, the product is even a $T_\Nwd$-space.
\end{proposition}

Since a finite product of regular open sets is a regular open set, we have the following.

\begin{proposition}
	Any finite product of $T_\ClosedMeetsRO$-spaces is a $T_\ClosedMeetsRO$-space.
	Any finite product of $T_\NwdOrRO$-spaces is a $T_\NwdOrRO$-space.
\end{proposition}

\begin{example}
	Let $\kappa$ be an infinite cardinal. The Tychonoff product of $\kappa^*$-many copies of the $T_\ClosedOrRO$-space $\{-1, 0, 1\}$ from Example~\ref{e:Khalimsky} is not even a $T_{\kappa\mhyp\BP}$-space. Therefore, the bounds on the number of factors in the previous propositions are sharp.
\end{example}

A function $f:X\to Y$ between topological spaces is called {\em $\kappa$-Borel} if for any open (or equivalently $\kappa$-Borel) set $U\subset X$ the preimage $f^{-1}(U)$ belongs to the algebra $\kappa\mhyp\Borel(X)$.

\begin{proposition} Let $\kappa$ be an infinite cardinal. A topological space $X$
is $T_{\kappa\mhyp\Borel}$ if it admits a $\kappa$-Borel function $f:X\to Y$
into a $T_{\kappa\mhyp\Borel}$-space $Y$ such that for every $y\in Y$ the
preimage $f^{-1}(y)$ is a $T_{\kappa\mhyp\Borel}$-space.
\end{proposition}

\begin{proof} By Remark~\ref{t:regular} we can assume that $\kappa$ is regular. Given any point $x\in X$, consider the point $y = f(x)\in Y$. By Theorem~\ref{t:Borel}, the singleton $\{y\}$ can be written as the intersection $\{y\}=F\cap G$ of a closed set $F\subset Y$ and $G_{<\kappa}$-set $G\subset Y$. Since $f$ is $\kappa$-Borel, the preimage $f^{-1}(y)=f^{-1}(F)\cap f^{-1}(G)$ belongs to the algebra $\kappa\mhyp\Borel(X)$. Applying Theorem~\ref{t:Borel} to the $T_{\kappa\mhyp\Borel}$-space $f^{-1}(y)$, we can find a closed set $F'$ in $X$ and a $G_{<\kappa}$-set $G'$ in $X$ such that $\{x\}=f^{-1}(y)\cap F'\cap G'$. Observe that the sets $f^{-1}(y)$, $F'$, $G'$ belong to the algebra $\kappa\mhyp\Borel(X)$, which implies that $\{x\}=f^{-1}(y)\cap F'\cap G'\in\kappa\mhyp\Borel(X)$ and $X$ is a $T_{\kappa\mhyp\Borel}$-space.
\end{proof}

\begin{proposition} Let $\kappa$ be an infinite cardinal. A topological space $X$ is a $T_{\kappa\mhyp\BP}$-space if it admits a function $f:X\to Y$ into a $T_{\kappa\mhyp\BP}$-space $Y$ such that
\begin{itemize}
\item for every $y\in Y$ the preimage $f^{-1}(y)$ is a $T_{\kappa\mhyp\BP}$-space;
\item for every open set $U\subset Y$ the pre-image $f^{-1}(U)$ belongs to the algebra $\kappa\mhyp\BP(X)$;
\item for any nowhere dense set $N\subset Y$ each singleton $\{x\}\subset f^{-1}(N)$ is nowhere dense in $X$.
\end{itemize}
\end{proposition}

\begin{proof} By Remark~\ref{t:regular} we can assume that the cardinal $\kappa$ is regular. Given any point $x\in X$, consider the point $y = f(x)\in Y$. By Theorem~\ref{t:BP}, the singleton $\{y\}$ is either nowhere dense of a $G_{<\kappa}$-set in $Y$. If $\{y\}$ is nowhere dense in $Y$, then by the third condition, the singleton $\{x\}\subset f^{-1}(y)$ is nowhere dense in $X$ and hence belongs to the algebra $\kappa\mhyp\BP(X)$.

So, assume that $\{y\}$ is a $G_{<\kappa}$-set in $X$. The second condition guarantees that the preimage $f^{-1}(y)$ belongs to the algebra $\kappa\mhyp\BP(X)$.
By the first condition, the singleton $\{x\}$ belongs to the algebra $\kappa\mhyp\BP(f^{-1}(y))$. So by Theorem~\ref{t:BP}, either $\{x\}$ is nowhere dense in $f^{-1}(y)$ and hence in $X$, or there is a $G_{<\kappa}$-set $G\subset X$ such that $\{x\} = G\cap f^{-1}(y)$. In both cases we may conclude that $\{x\}\in\kappa\mhyp\BP(X)$.
\end{proof}

\section{More connections with other separation axioms}
\label{s:connections}

We have observed that several axioms based on Borel and Baire algebras are equivalent to some classical separation axioms. Clearly, $T_\Closed$ is $T_1$, $T_\ClosedOrOpen$ is $T_\frac{1}{2}$ (or $T_{ES}$), $T_\Constructible$ is equivalent to $T_D$, and $T_{\infty\mhyp\Borel}$ is equivalent to $T_0$. Also note that $T_\ClosedOrG{\delta}$-spaces and $T_\Borel$-spaces are called $GT_{\frac12}$-spaces and $GT_D$-spaces, respectively, by Lo in \cite{Lo_01}. In the last section we observe more connections with other separation axioms.

Let us say that in a topological space $X$ a set $A$ \emph{is separated from}
a set $B$ if there is an open neighborhood of $A$ disjoint with $B$. We also
identify a point $x \in X$ with the singleton $\{x\}$ when using this notion.
In \cite{T_1/4} Arenas, Dontchev, and Ganster defined a topological space $X$
to be $T_\frac{1}{4}$ if for every finite set $F\subset X$ and every point
$x\in X\setminus F$ either $x$ is separated from $F$ or $F$ is separated from
$x$. This property was considered earlier by Aull and Thron in
\cite{Aull_Thron_62} under name $T_F$. It turns out that the property is
equivalent to $T_\ClosedOrG{\infty}$. We include a proof for completeness.

\begin{proposition}
	A topological space $X$ is $T_\frac{1}{4}$ if and only if it is $T_\ClosedOrG{\infty}$.
	\begin{proof}
		Let $F$ be a finite subset of $X$ and $x \in X\setminus F$ be a point.
 If the singleton $\{x\}$ is
closed, then $F$ is separated from $x$. If the singleton $\{x\}$ is an
intersection of open sets, then it is separated from every point of $F$ and
hence from $F$ since it is finite.
		On the other hand, if $X$ is $T_\frac{1}{4}$, then it is $T_\ClosedOrG{\infty}$. Otherwise, there are points $x, y, z \in X$ such that $x$ is not separated from $y$, and $z$ is not separated from $x$. Therefore, $x$ is not separated from $\{y, z\}$, and $\{y, z\}$ is not separated from $x$.
	\end{proof}
\end{proposition}

There are several separation axioms associated with the $\alpha$-topology (see for example \cite{alpha-separation}): for $i \in \{0, D, \frac{1}{2}, 1\}$ a topological space $(X, \tau)$ is \emph{$_\alpha T_i$} if the induced space $(X, \tau^\alpha)$ is $T_i$. Proposition~\ref{t:alpha} shows that $_\alpha T_1$ is equivalent to $T_\ClosedOrNwd$; $_\alpha T_\frac{1}{2}$ and $_\alpha T_D$ are equivalent to $T_{\w\mhyp\BP}$, and $_\alpha T_0$ is equivalent to $T_{\infty\mhyp\BP}$.
There is also a notion of \emph{feebly open} sets and the corresponding axioms
\emph{feebly $T_0$} and \emph{feebly $T_1$}. However, feebly open sets are
precisely $\alpha$-open sets \cite{semi-separation}, so the axioms \emph{feebly
$T_0$} and \emph{feebly $T_1$} are equivalent to the axioms $_\alpha T_0$ and
$_\alpha T_1$, respectively.

The axioms $_\alpha T_i$ can be equivalently defined using the classical
definitions of $T_i$ where open are sets replaced with $\alpha$-open sets and
closed are replaced with $\alpha$-closed sets. When semi-open sets are used
instead, we obtain the corresponding axioms \emph{semi-$T_i$}. Note that the
family of all semi-open sets is closed under unions, but not under finite
intersections, so it does not form a topology in general. It turns out that
semi-$T_i$ is equivalent to $_\alpha T_i$ for $i \in \{0, D, \frac{1}{2}\}$
\cite{semi-separation}, \cite{alpha-separation}, but semi-$T_1$ is strictly weaker
than $_\alpha T_1$. In fact, it is equivalent to $T_\NwdOrRO$ \cite{semi-separation}.

Originally, Levine defined $T_\frac{1}{2}$-spaces by the condition that every generalized closed set is closed \cite{Levine_70}. Later, Dontchev and Ganster defined $T_\frac{3}{4}$-spaces by the condition that every generalized $\delta$-closed set is $\delta$-closed \cite{T_3/4}. Here the $\delta$-topology stands for the semi-regularization topology. They also proved that a topological space is $T_\frac{3}{4}$ if and only if every singleton is closed or regular open, i.e.\ if it is $T_\ClosedOrRO$.

The equivalences of semi-$T_1$ and $T_\frac34$ with $T_\NwdOrRO$ and $T_\ClosedOrRO$, respectively, were the motivation for introducing $T_\ClosedOrRO$, $T_\ClosedMeetsRO$, and $T_\NwdOrRO$ in the first place. We do not know whether the property $T_\ClosedMeetsRO$ is equivalent to any separation axiom considered before.

We conclude with a copy of Diagram~\ref{diag_2} containing the names of equivalent separation axioms:
%\begin{equation*}
%\xymatrix{
	%*+\txt{$T_1$, \\ $T_\Closed$} \ar@{=>}[r] \ar@{=>}[dd] &
	%*+\txt{$T_\frac{3}{4}$, \\ $T_\ClosedOrRO$} \ar@{=>}[r] \ar@{=>}[d] &
	%*+\txt{$T_{\frac12}$, $T_{ES}$, \\ $T_{\ClosedOrOpen}$, \\ $T_\ClosedOrG{<\w}$} \ar@{=>}[r] \ar@{=>}[d] &
	%*+\txt{$GT_{\frac12}$, \\ $T_\ClosedOrG{\delta}$, \\ $T_\ClosedOrG{<\w_1}$} \ar@{=>}[r] \ar@{=>}[d] &
	%T_\ClosedOrG{<\kappa} \ar@{=>}[r] \ar@{=>}[d] &
	%*+\txt{$T_\frac{1}{4}$, $T_F$, \\ $T_\ClosedOrG{\infty}$} \ar@{=>}[d] &
	%\\
	%&
	%*+\txt{$T_\ClosedMeetsRO$} \ar@{=>}[r] \ar@{=>}[d]&
	%*+\txt{$T_D$, \\ $T_\Constructible$, $T_{\w\mhyp\Borel}$} \ar@{=>}[r] \ar@{=>}[d] &
	%*+\txt{$GT_D$, \\ $T_\Borel$, $T_{\w_1\mhyp\Borel}$} \ar@{=>}[r] \ar@{=>}[d] &
	%T_{\kappa\mhyp\Borel} \ar@{=>}[r] \ar@{=>}[d] &
	%*+\txt{$T_0$, \\ $T_{\infty\mhyp\Borel}$} \ar@{=>}[d] &
	%\\
	%*+\txt{$_\alpha T_1$, \\ $T_\ClosedOrNwd$} \ar@{=>}[r] &
	%*+\txt{semi-$T_1$, \\ $T_\NwdOrRO$} \ar@{=>}[r] &
	%*+\txt{$_\alpha T_{\frac12}$, $_\alpha T_D$, \\ semi-$T_\frac{1}{2}$, semi-$T_D$, \\ $T_\OpenOrNwd$, $T_{\w\mhyp\BP}$} \ar@{=>}[r] &
	%*+\txt{$T_\BP$, $T_{\w_1\mhyp\BP}$} \ar@{=>}[r] &
	%T_{\kappa\mhyp\BP} \ar@{=>}[r] &
	%*+\txt{$_\alpha T_0$, \\ semi-$T_0$, \\ $T_{\infty\mhyp\BP}$} &
%}
%\end{equation*}

% \clearpage % add if needed

\begin{equation*}
\xymatrix{
	*+\txt{$T_1$, \\ $T_\Closed$} \ar@{=>}[d] \ar@{=>}[rr] &&
	*+\txt{$_\alpha T_1$, \\ $T_\ClosedOrNwd$} \ar@{=>}[d] &
	\\
	*+\txt{$T_\frac{3}{4}$, \\ $T_\ClosedOrRO$} \ar@{=>}[d] \ar@{=>}[r] &
	*+\txt{$T_\ClosedMeetsRO$} \ar@{=>}[d] \ar@{=>}[r]&
	*+\txt{semi-$T_1$, \\ $T_\NwdOrRO$} \ar@{=>}[d] &
	\\
	*+\txt{$T_{\frac12}$, $T_{ES}$, \\ $T_{\ClosedOrOpen}$, \\ $T_\ClosedOrG{<\w}$} \ar@{=>}[r] \ar@{=>}[d] &
	*+\txt{$T_D$, $T_\Constructible$, \\ $T_{\w\mhyp\Borel}$} \ar@{=>}[d] \ar@{=>}[r] &
	*+\txt{$_\alpha T_{\frac12}$, semi-$T_\frac{1}{2}$, \\ $_\alpha T_D$, semi-$T_D$, \\ $T_\OpenOrNwd$, $T_{\w\mhyp\BP}$} \ar@{=>}[d] &
	\\
	*+\txt{$GT_{\frac12}$, $T_\ClosedOrG{\delta}$, \\ $T_\ClosedOrG{<\w_1}$} \ar@{=>}[r] \ar@{=>}[d] &
	*+\txt{$GT_D$, $T_\Borel$, \\ $T_{\w_1\mhyp\Borel}$} \ar@{=>}[r] \ar@{=>}[d] &
	*+\txt{$T_\BP$, \\ $T_{\w_1\mhyp\BP}$} \ar@{=>}[d] &
	\\
	T_\ClosedOrG{<\kappa} \ar@{=>}[r] \ar@{=>}[d] &
	T_{\kappa\mhyp\Borel} \ar@{=>}[r] \ar@{=>}[d] &
	T_{\kappa\mhyp\BP} \ar@{=>}[d] &
	\\
	*+\txt{$T_\frac{1}{4}$, $T_F$, \\ $T_\ClosedOrG{\infty}$} \ar@{=>}[r] &
	*+\txt{$T_0$, \\ $T_{\infty\mhyp\Borel}$} \ar@{=>}[r] &
	*+\txt{$_\alpha T_0$, semi-$T_0$, \\ $T_{\infty\mhyp\BP}$} &
}
\end{equation*}

\end{document}